\documentclass[10pt]{amsart}
\usepackage{sty}
\usepackage[foot]{amsaddr}
\begin{document}

\title{Greedy heuristics and linear relaxations for the random hitting set problem}

\author{Gabriel Arpino$^1$}
\address{\hspace{-1.0em}$^1$University of Cambridge, CB2 1PZ Cambridge, UK}
\email{ga442@cam.ac.uk}

\author{Daniil Dmitriev$^2$}
\address{$^2$ETH {Z\"u}rich and ETH AI Center, 8092 Zürich, Switzerland}
\email{daniil.dmitriev@ai.ethz.ch}

\author{Nicolo Grometto$^3$}
\address{$^3$Princeton University, Sherrerd Hall, Princeton, NJ 08540, USA}
\email{ng1069@princeton.edu}


\date{\today}

\begin{abstract}
Consider the \emph{Hitting Set} problem where, for a given universe $\mathcal{X} = \curly{1\hdots, n}$ and a collection of subsets $\mathcal{S}_1,\hdots, \mathcal{S}_m$, one seeks to identify the smallest subset of $\mathcal{X}$ which has nonempty intersection with every element in the collection. 
We study a probabilistic formulation of this problem, where the underlying subsets are formed by including each element of the universe with probability $p$, independently of one another. For large enough values of $n$, we rigorously analyse the average case performance of Lovász’s celebrated greedy algorithm \cite{lovasz1975ratio} with respect to the chosen input distribution. In addition, we study integrality gaps between linear programming and integer programming solutions of the problem.
\end{abstract}

\maketitle



\section{Introduction}\label{sec:intro}


$\hs$ is a classical problem in combinatorial optimization which, for a given ground set $\mathcal{X}:=\curly{1,...,n}$ of elements and a collection $\mathcal{C} := \curly{\mathcal{S}_1,...,\mathcal{S}_m}$ of subsets of $\mathcal{X}$, asks to identify the smallest set $\mathcal{S}\subseteq \mathcal{X}$ that intersects every subset in $\mathcal{C}$. $\hs$ arises naturally from the study of \emph{Minimum Vertex Covers on Hypergraphs} ($\mvch$), upon viewing hyperedges as subsets and vertices as elements of the ground set. 
In addition, it is dual to the \emph{Set Cover} problem \cite{paschos1997survey}, which has a rich history in computational complexity theory, including appearing as one of Karp's 21 NP-complete problems. The aim of the present work is to expand our understanding of the hardness of typical $\hs$ instances, by analysing its average-case complexity with respect to a large class of random inputs generated by assigning each element of the ground set to any subset with fixed probability $p$, independently. To further delineate the purpose of this work, let us recall that $\hs$ has the following integer programming ($\ip$) formulation,
\begin{equation}
\label{def:ip}
\valip:=  
\begin{cases}
\begin{aligned}
& \underset{\x}{\text{minimize}}
& & \|\x\|_1 \\
& \text{subject to}
& &  \A\x \geq \mathbf{1} , \: \x \in \left \{0,1\right\}^n,
\end{aligned}
\end{cases}
\end{equation}
where the $i$-th row of $\A \in \curly{0,1}^{m\times n}$ provides a binary encoding of the membership of the elements of $\mathcal{X}$ in the set $\mathcal{S}_i$ and $\mathbf{1} := (1, \ldots, 1) \in \mathbb{R}^m$. With the vertex cover formulation of the problem at hand, we note that $\A$ consists of the incidence matrix of the underlying hypergraph.
In particular, the first constraint in (\ref{def:ip}) ensures that each set in $\mathcal{C}$ is hit by a prescribed candidate solution vector. A natural convex relaxation is obtained by allowing fractional solutions, and may be expressed as the following linear program ($\lp$),
\begin{equation}
\label{def:lp}
\vallp:=
\begin{cases}
\begin{aligned}
& \underset{\x}{\text{minimize}}
& & \|\x\|_1 \\
& \text{subject to}
& & \A\x \geq \mathbf{1}, \:  \x \in [0,1]^n.
\end{aligned}
\end{cases}
\end{equation}
Whilst clearly $\vallp\leq \valip$, tightness need not hold in general. In fact, for $m = n$ and $\A \in \curly{0,1}^{n\times n}$ chosen such that each row and column contains exactly $k$ ones, for some fixed $1 < k < n$, an optimal solution is provided by $\x^*_{\texttt{LP}}=\left(1/k,...,1/k\right)$, which is not integral, thus leading to a strictly smaller objective whenever $n / k$ is not an integer. This evidences the existence of a significant \emph{integrality gap}, quantified by the ratio \(\valip / \vallp\). In \cite{lovasz1975ratio}, Lovász proved an essentially optimal worst-case upper bound on the $\hs$ multiplicative integrality gap of $1+\log \dmax$, where $\dmax$ corresponds to the maximum degree in the underlying hypergraph. 
This is obtained by analysing the \greedy algorithm (Algorithm~\ref{alg:greedy}), which constructs a vertex cover by sequentially adding vertices with the highest degree amongst the uncovered edges, and will be discussed in more detail in the next sections. However, in many natural examples, the maximum degree $\dmax$ grows with the number of vertices in the hypergraph, thus leading to progressively worse bounds for increasingly large hypergraphs. Besides being arguably the most natural candidate for solving $\hs$, the greedy algorithm has been shown to be essentially the best possible polynomial time approximation algorithm \cite{slavik1996tight} for the worst-case instances of this classical problem. In particular, the greedy algorithm has an approximation ratio of $O(\log m)$, thus finding coverings that may be at most $C\log m$ times as large as the minimum one, for a large enough constant $C > 0$. \\
\noindent 
Despite extensive work conducted on $\hs$ in the last decades, a gap remains in our understanding of the typical performance of linear programming and the greedy algorithm on random problem instances. We hence pose the following questions:
\begin{enumerate}
    \item For which regimes of $m,n,p$ are there (multiplicative) integrality gaps in random instances of $\hs$?
    \item What is the average case performance of the greedy algorithm with respect to a chosen input distribution? 
\end{enumerate}
In the present work, we provide answers to the above questions 
\emph{with high probability} (w.h.p.) in a non-asymptotic sense, 
in the setting where the cardinality $n$ of the ground set $\mathcal{X}$ 
is large but finite. We will prove the absence of integrality gaps in a wide regime of $n,m,p$, by conducting an average case analysis of an algorithm that outputs integral covers of matching size to the fractional ones. In addition, a rigorous analysis of the greedy routine will follow by a straightforward reduction. The forthcoming results are valid under the conditions listed below, which will be assumed to hold throughout.
\begin{assumption}\label{asmpt:A_np_m}We assume that
\begin{enumerate}[ref={\theassumption.\arabic*}]
    \item Each element $j\in \mathcal{X}$ is assigned to any subset $\mathcal{S}_i$, $i \in [m]$ with probability $p \equiv p(n)$, independently. That is, $\A \in \{0, 1\}^{m \times n}$ is such that $A_{ij} \overset{\text{iid}}{\sim} \mathrm{Bernoulli}(p)$; 
    \item $n$ is intended to be large but finite;
    \item \(m \equiv m(n) = \mathrm{poly}(n)\), i.e. $\exists c, C > 0$, such that $c n^c \leq m \leq Cn^C$ for $n$ large enough; \label{asmpt:A_np_m.c}
    \item There exist $\delta \in (0,1)$, such that \(p \equiv p(n)\) satisfies \(1/n^{\delta} \leq p \leq 1/2\), for all \(n\) large enough.
\label{asmpt:A_np_m.d}
\end{enumerate}
\end{assumption}
\noindent
Note that in Assumption~\ref{asmpt:A_np_m.c}, the upper bound is chosen to avoid trivial solutions w.h.p., which e.g. arise in the setting where the number of sets grows exponentially in the cardinality of $\mathcal{X}$.  In addition, Assumption~\ref{asmpt:A_np_m.d} is by no means restrictive, since one can show that for $m = \text{poly}(n)$ and \(np \ll \log n\), we have that w.h.p., $\A$ contains an all-zero row, i.e. no feasible solution for $\ip$. The requirement \(p \leq 1/2\) is chosen for technical convenience and can be relaxed to any constant \(p\), encompassing the regime in~\cite{iliopoulos2021group}.\\

\noindent

%
\noindent
Our contributions stem from the study of the size of the inclusion sets $I_j := \curly{i \in [m]\: : \: j \in \mathcal{S}_i}$, for $j \in [n]$, which in the \(\mvch\) formulation of the problem at hand correspond to the degrees of vertices in the underlying hypergraph. The key quantity under scrutiny is the average inclusion set size, that is $\E |I_j| = mp$, for all $j$, under the present distributional assumptions. This quantity exhibits two separate regimes of interest, referred to as the \emph{sparse}, $mp \ll \log{n}$, and \emph{dense}, $mp \gg \log{n}$, regimes. These in turn determine the size of the maximum inclusion set, or maximum degree, $\dmax := \max_{j \in [n]} |I_j|$. We characterize the integrality gap behaviour up to multiplicative constants and analyse Lovász's $\greedy$ algorithm \cite{lovasz1975ratio} in these two regimes w.h.p as $n \to \infty$. We do this by proving the success of a simple greedy heuristic, the \bgreedy algorithm (Algorithm \ref{alg:block_greedy}). Throughout, we use the notation $\text{val}_{\mathtt{Gr}}$, $\valbgr$ to denote the size of the hitting set returned by $\greedy$ and $\bgreedy$ respectively. Below we provide an informal description of the main results, also depicted in Figure \ref{fig:phase_diagram}. The formal statements are given in Theorem~\ref{thm:alg_solution_ub} and Theorem~\ref{thm:greedy}.
\\

\noindent
\paragraph{\textbf{Sparse Regime} ($mp \ll \log{n}$):}
We show that there is \emph{no integrality gap} ($\vallp/\valip \sim 1$) in the sparse regime by proving that the \bgreedy algorithm succeeds in reaching the \lp lower bound of $\frac{m}{\dmax}$.
\[ \valbgr \sim \valip \sim \vallp \sim \frac{m}{\dmax}. \]

\vspace{0.3em}

\paragraph{\textbf{Dense Regime} ($mp \gg \log{n}$):}
We prove the existence of a tight, non-vanishing integrality gap between \ip and \lp in the dense regime through the first moment method. We show that the \bgreedy algorithm performs as well as \ip in this regime, i.e.
\[ \frac{1}{p} \log\left(\frac{mp}{\log n}\right) \sim \valbgr \sim \valip \gg \vallp \sim \frac{1}{p} \sim \frac{m}{\dmax}. \] 

\vspace{0.3em}

\paragraph{\textbf{Threshold Regime} ($mp \sim \log{n}$):} This regime smoothly interpolates between the sparse and dense ones, with no integrality gaps. The scaling for all quantities of interest is $m / \dmax \sim 1/p$.

\vspace{0.3em}

\paragraph{\textbf{Greedy}:} We prove that $\valgr \sim \valip$ for $\delta < 1/2$, where \(\delta\) is the parameter from Assumption~\ref{asmpt:A_np_m.d}. 
    
\vspace{1em}
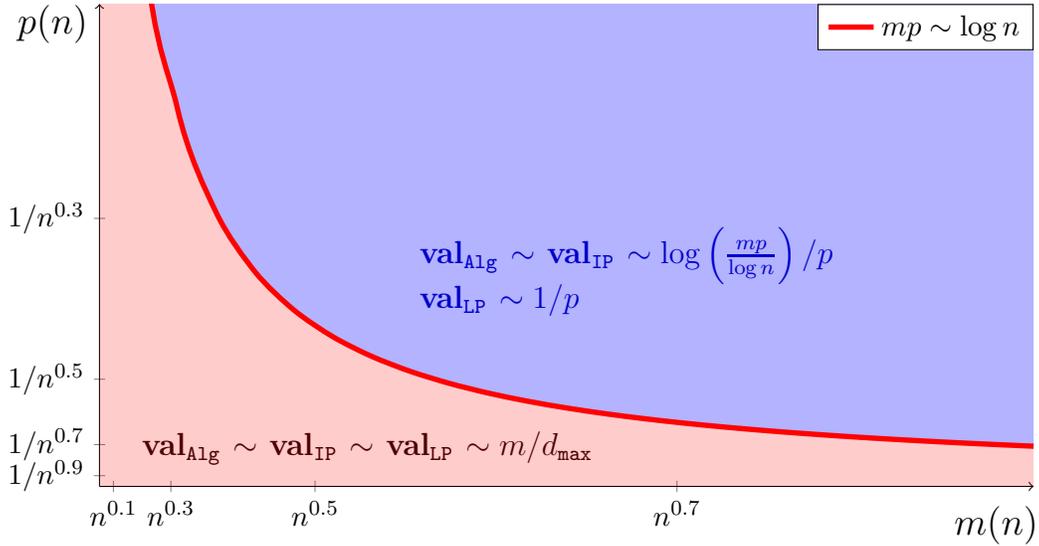
\begin{figure}
    \centering
  \begin{tikzpicture}
    \begin{axis}[ymin = 0, ymax = 0.45, axis lines=middle, height=8cm, width=14cm,
    legend style={
      at={(1,1)}
      },
    axis line style={->},
    x label style={at={(axis description cs:0.96,-0.14)},anchor=south},
    y label style={at={(axis description cs:-0.05,0.9)},anchor=south},
    xlabel={\LARGE $m(n)$},
    ylabel={\LARGE $p(n)$},
    xtick={1.58, 3.98, 10, 25.11},
    xticklabels={\large $n^{0.1}$,
                \large $n^{0.3}$,
                \large $n^{0.5}$, 
                \large $n^{0.7}$},
    ytick={0.01,0.039,0.1,0.25,0.63},
    yticklabels={\large $1 / n^{0.9}$,
                \large $1 / n^{0.7}$,
                \large $1 / n^{0.5}$,
                \large $1 / n^{0.3}$,
                \large $1 / n^{0.1}$},]
        \addplot[thick, color=red, smooth, name path = inv, domain = 1:40, line width=2pt] {1.5/x};
        \addplot[name path = floor, draw = none] coordinates {(1,0) (40,0)};
        \addplot[fill=red, 
        fill opacity=0.2] fill between[of = inv and floor];
        \path[name path=axis] (axis cs:0,40) -- (axis cs:40,40);
        \addplot[fill = blue, fill opacity=0.3] fill between[of= inv and axis];
        \legend{\large $ mp \sim \log n$}
    \node[below, text=red!30!black,font=\bfseries\Large] at (112, 6) {\valbgr \(\sim\) \valip \(\sim\) \vallp \(\sim m / \dmax\)};
    \node[above, text=blue!80!black, font=\bfseries\Large] at (220,18) {\valbgr \(\sim\) \valip \(\sim \log\left(\frac{mp}{\log n}\right) / p \)};
    \node[above, text=blue!80!black, font=\bfseries\Large] at (167,15) {\vallp \(\sim 1/p\)};
\end{axis}
\end{tikzpicture}
\caption{Transition between the sparse and the dense regime for different values of the average inclusion set size $mp$.}
    \label{fig:phase_diagram}
\end{figure}

\noindent
The rest of the paper is organized as follows. In Section \ref{sec:notation}, we present relevant notation. In Section \ref{sec:related work}, we outline and discuss the relevant literature. In Section \ref{sec:bounds_LP_IP}, we prove a number of preliminary results that will be instrumental in developing the core arguments. In Section \ref{sec:d_max}, we study the value of $\dmax$. Subsequently, in Section~\ref{sec:algo_sol}, we delve into the algorithmic aspects of the problem at hand by first providing guarantees on a simple algorithm, \bgreedy. We then analyse \greedy by means of a reduction. We conclude in Section \ref{sec:discussion} by summarizing the results and offering indications for future work. We defer the proofs of more technical results to the appendix, in order to streamline the presentation for the reader's convenience.

\section{Notation and conventions}\label{sec:notation}
For integers $k \in \mathbb{N}$, we write $[k]:=\curly{1,...,k}$. We denote vectors, matrices by bold-faced Roman letters $\mathbf{x}, \A \in \mathbb{R}^k, \mathbb{R}^{k \times k}$, respectively, for some $k \in \mathbb{N}$. Define the \emph{inclusion set} of an element, or node, $j \in [n]$ as $I_j = \left\{i \in [m] : j \in \mathcal{S}_i \right\}$. We denote the $\ell_1$ norm of the $j$-th column of $\A$ by $X_j$, $j\in [n]$, noting that $X_j = |I_j|$ and $X_1, \ldots X_n \overset{\text{iid}}{\sim} \mathrm{Binomial}(m, p)$. In addition, we let $\dmax \equiv \dmax(X_1, \ldots, X_n) := \max_{i \in [n]} X_i$. We use $\E, \text{Var}$ to denote expectation and variance, respectively. By $\lesssim$, $\gtrsim$ we denote inequalities up to constants. We let $A \sim B$ denote that $A \lesssim B \lesssim A$ for large enough $n$. We let $\log$ denote the natural logarithm. \\ For possibly random functions \(f(n), g(n)\), we let \(\{f \lesssim g\}\) denote a sequence of events \(\{f(n) \leq A g(n)\}\) for some constant \(A > 0\) independent of \(n\). Consequently, \(\P(f \lesssim g)\) is viewed as a function of \(n\). The notation for other inequalities is defined analogously. 
\noindent
We say that a sequence of events $\curly{A_n}$ holds \emph{with high probability} (w.h.p.) with respect to a probability measure $\mathbb{P}$ if there exists a constant $c > 0$, independent of $n$, such that $\mathbb{P}(A_n) \geq 1-n^{-c}$, for large enough values of $n$.

\section{Related Work}\label{sec:related work}
Perhaps the most well-known algorithm for solving $\hs$, or equivalently $\mvch$, is the greedy algorithm of Lovász~\cite{lovasz1975ratio}, with runtime complexity $O(mn^2)$. This algorithm, which constructs a cover by sequentially adding elements of the ground set which hit the largest number of remaining subsets, was initially studied by Lovász \cite{lovasz1975ratio} and Johnson \cite{johnson1973approximation} independently, for deterministic hypergraphs. Lovász analyses the greedy algorithm to obtain an upper bound on the $\hs$ integrality gap of $1+\log \dmax$. However, for many natural examples, the maximum degree grows with the number of vertices in the hypergraph, leading to progressively worse bounds as the size increases. Slavik \cite{slavik1996tight} developed the tightest known analysis, showing that the greedy algorithm is essentially the best-possible polynomial-time approximation algorithm for set cover, with an approximation ratio of $O(\log m)$.

\noindent
Nevertheless, much less is known about the typical performance of polynomial-time algorithms on random instances of $\hs$. Closing this gap is important from a theoretical standpoint and for applications in combinatorial inference. A prime example of this is found in \emph{group testing}, a classical inference problem where one aims to identify a small subset of defective items within a large population by conducting the smallest number of pooled tests, with applications ranging from the analysis of communication protocols \cite{fernandez2013unbounded} to DNA sequencing \cite{erlich2015biological} and search problems \cite{du2000combinatorial}. In \cite{iliopoulos2021group}, Iliopulos and Zadik consider the smallest hitting set as an estimator in the setting of the group testing problem, referring to it as the \emph{Smallest Satisfying Set} estimator. In particular, they provide extensive empirical evidence supporting the claim that the class of instances of the random hitting set problem induced by non-adaptive group testing is tractably solvable by computers. \\
\noindent
The analysis of a random instance of $\hs$ appears in the work of M\'ezard and Tarzia and relies on nonrigorous techniques from statistical physics \cite{mezard2007statistical}. This work considers regular uniform hypergraphs, where the degree of vertices and the size of edges are fixed and assumed to be constant. Depending on these values, they showed phase transitions between the replica symmetry, 1-replica symmetry breaking, and full replica symmetry breaking phases, which characterize the complexity of the optimization landscape for this problem in the average case setting.\\
\noindent
Another instance was studied by Telelis and Zissimopoulos \cite{telelis2005absolute} in the setting of random Bernoulli hypergraphs, where elements belong to subsets independently with \emph{fixed} probability $p\in(0,1)$. Their analysis concerns the asymptotic regime where the size $n$ of the ground set scales to infinity. In this setting, they study the average-case performance of a simple deterministic algorithm which approximates random $\hs$ within an additive error term of order at most $\log m$ almost everywhere. This gives an improvement over Lovász's argument in~\cite{lovasz1975ratio}, since in this setting $\dmax$ is at most of order $m$. However, the analysis in~\cite{telelis2005absolute} does not capture the case of sparse hypergraphs obtained for $p \to 0$ as $n \to \infty$ and, more broadly, any regime where $p$ scales with the size of the ground set $n$. The analysis in~\cite{telelis2005absolute} also does not prove guarantees for the $\greedy$ algorithm in the chosen parameter regime. \\ 
\noindent
For completeness, we also bring to the reader’s attention a more recent line of work \cite{borst2022gaussian, borst2022integrality}, where the authors obtain bounds on (additive) integrality gaps between the value of a random integer program $\max \c^T \x, \A\x \leq \b, \x \in \left\{0,1\right\}^n$ with $m$ constraints and that of its linear programming relaxation for a wide range of distributions on $(\A,\b,\c)$, holding w.h.p. as $n\rightarrow \infty$. These include the case where the entries of $\A$ are uniformly distributed on an integer interval consisting of at least three elements and where the columns of $\A$ are distributed according to an isotropic logconcave distribution. However, these fail to capture the setting where $\A$ is sparse with entries in $\left\{0,1\right\}$, which is of interest for $\hs$.

\section{Bounds on \vallp and \valip}\label{sec:bounds_LP_IP}
\label{sec:bounds}
The aim of this section is to obtain preliminary bounds on $\vallp, \valip$, starting from the following deterministic lower bound on $\vallp$, which holds across all regimes of $m,n,p$.
\begin{lemma}
\label{lemma:lp_lb} We have that
\[\vallp \geq \frac{m}{\dmax}.\]
\end{lemma}
\begin{proof}
Let $\x^*_{\texttt{LP}} = (\x_1^*, \x_2^*, \ldots, \x_m^*)$ be an optimal solution for (\ref{def:lp}). Since $\A \x^*_{\texttt{LP}} \geq \mathbf{1}$ entrywise, by summing all entries we obtain that  
\[
m \leq \sum_{i} \x^*_i X_i \leq \dmax\sum_i \x^*_i = \dmax \vallp.
\]
which upon rearranging yields the desired result.
\end{proof}
\noindent
In addition to the above, we have the following elementary upper bound on \vallp, which holds both in the sparse and dense regime. 
\begin{lemma}
\label{lemma:lp_ub}
    There exists $c > 0$, independent of $n$, such that
    \begin{equation*}
    \P\left(\vallp \lesssim \frac{1}{p}\right) \geq 1 - \exp\left(-cn^{1-\delta}\right).
\end{equation*}
This also implies that \(\P(\ip \text{ is feasible}) \geq 1 - \exp\left(-cn^{1-\delta}\right).\)
\end{lemma}
\begin{proof}
    Consider the candidate feasible solution \(\hat{\x} := \frac{1}{\tilde{C} np} \mathbf{1}\), for some constant $0<\tilde{C}<1$. The following results from applying a union bound over constraints and the standard Chernoff bound.
    \begin{align*}
        \P\round{\hat{\x} \: \text{not feasible}} & = \P\round{\exists i \in [m]: \round{\A \hat{\x}}_i < 1} \\
        & \leq m \P\round{\text{Bin}(n,p)< \tilde{C}np} \\
        & \leq n^C \exp\round{-\frac{(1-\tilde{C})^2np}{2}}\\
        & \leq \exp\round{-c n^{1-\delta}}.
    \end{align*}
The desired conclusion follows by considering the complementary event to the one above and noting that $\left \|\hat{\x}\right\|_1 \sim 1/p$. Note that the event \(\{\hat{\x} \text{ is feasible for } \lp\}\) implies the event \(\{\ip \text{ is feasible}\}\).
\end{proof}

\noindent
Lemma~\ref{lemma:lp_lb} clearly implies that \(\valip \geq m / \dmax\). However, when \(mp \gg \log n\), this lower bound is not precise. Indeed, we will apply the first moment method to prove a tighter lower bound on $\valip$. We first recall the following properties of the Lambert \(W\) function, which consists of the solution to the equation $ye^y = x$, for $x \geq 0$. 
\begin{lemma}[Lambert \(W\) function, \cite{hoorfar2008inequalities}]\label{lambert}
For any $x\geq e$, there holds that
\begin{align}\label{lambert_bounds}
    \log x - \log\log x + \frac{\log \log x}{2\log x}\leq W_0(x) \leq \log x - \log\log x + \frac{e}{e-1}\frac{\log \log x}{\log x}.
\end{align}
In particular,
\begin{align}\label{lambert_asymptotics}
    W_0(x) = \log x - \log \log x + o(1), \quad \text{as} \quad x \rightarrow \infty.
\end{align}
In addition, for any $x \geq 1/e$, the following identity is satisfied
\begin{align}\label{lambert_variational}
    W_0(x) = \log \frac{x}{W_0(x)}.
\end{align}
\end{lemma}
\noindent
We now provide a tighter lower bound for $\valip$ in the regime \(mp \gg \log{n}\) through the first moment method. 
\begin{lemma}
\label{lemma:first_moment_mp_large}
If \(mp \gg \log n\), we have that, for any \(D \geq 1\), for \(n\) large enough, 
\begin{equation*}
        \P\left(\valip \gtrsim \frac{1}{p} \log \left(\frac{mp}{\log n}\right)\right) \geq 1 - n^{-D}.
\end{equation*}
\end{lemma}
\begin{proof}[Proof of Lemma~\ref{lemma:first_moment_mp_large}]
Fix \(D \geq 1\).
Let $Z_k := \lvert\left\{  \x \in \left \{0,1\right\}^m : \A\x \geq \mathbf{1}, \|\x\|_1 = k \right \}\rvert$ be the number of feasible solutions of norm exactly $k$. Clearly, \(Z_k \leq Z_{k+1}\) for any \(k \geq 0\). We also have that
\begin{align*}
    \mathbb{E}Z_k & = \sum_{\|\x\| = k}\P\left({(\A\x)_i \geq 1, \forall i\in [m]}\right) = \binom{n}{k}\left(1-(1-p)^k\right)^m.
\end{align*}
We will now show that for \(k \ll \frac{1}{p}\log\left(\frac{mp}{\log n}\right)\), we have \(\mathbb{E} Z_{k} \leq n^{-D}\). Using that \(p \leq 1/2\) from Assumption~\ref{asmpt:A_np_m.d} and that for \(x \in (0, \frac{1}{2})\), we have \((1 - x)^y \geq e^{-2xy}\), we can bound
    \begin{equation*}
    \begin{aligned}
        \mathbb{E} Z_k &= \binom{n}{k} (1 - (1 - p)^k)^m \leq n^k \left(1 - e^{-2pk}\right)^m \\
        &\leq n^k e^{-m e^{-2pk}} = \exp \left\{k \log n - m e^{-2pk}\right\}.
    \end{aligned}
    \end{equation*}
Therefore, \(\mathbb{E}Z_k \leq n^{-D}\) will follow from 
\begin{equation}
\label{eq:fmm1}
    2pk e^{2pk} \leq -2Dpe^{2pk} + \frac{2mp}{\log n}.
\end{equation}
Since \(k \ll \frac{1}{p} \log \left(\frac{mp}{\log n}\right)\), we also have that \(k \leq k_{\ast} \coloneqq \frac{1}{2p} W_0 \left(\frac{mp}{D \log n}\right) \) for \(n\) large enough. For \(k = k_{\ast}\),
the left hand side of~(\ref{eq:fmm1}) is equal to \(\frac{mp}{D \log n}\), while the right hand side is lower bounded by \(\frac{mp}{\log n}\). Since \(D \geq 1\), we recover that \(\mathbb{E}Z_k \leq n^{-D}\). Note that for \(n\) large enough, \(\valip \ll \frac 1 p \log \frac{mp}{\log n}\) implies that \(Z_{k_{\ast}} > 0\). Therefore, applying Markov's inequality, we get that 
\begin{equation}
\begin{aligned}
\P\left(\valip \ll \frac{1}{p} \log \left(\frac{mp}{\log n}\right)\right) \leq \P\left(Z_{k_{\ast}} > 0\right) \leq \mathbb{E} Z_{k_{\ast}} \leq n^{-D},
\end{aligned}
\end{equation}
and the proof follows by considering the complementary events. Note that using similar derivations, one can also show that for \(k^{\ast} \coloneqq \frac{1}{p}\log\left(\frac{1}{\delta}\frac{mp}{\log n}\right)\), where \(\delta\) is defined in Assumption~\ref{asmpt:A_np_m.d}, we have \(\mathbb{E} Z_{k^{\ast}} \geq 1\).
\end{proof}

\section{Analysis of \(\dmax\)}\label{sec:d_max}
In this section we will first estimate the size of $\E \dmax$ and subsequently show that w.h.p. $\dmax \lesssim \E \dmax$, which suffices for the purposes of this work. For additional arguments concerning concentration of $\dmax$ around its mean, we refer the reader to Lemmas  \ref{lemma:d_max_concentration}, \ref{lemma:tensorization_var}.\\
\noindent
In order to deal with the more delicate sparse regime where $mp \ll \log n$, we state the following technical lemma, whose proof is presented in Appendix~B. 
\begin{lemma}
\label{lemma:mp<logn_conquering_logn}
For $mp \ll \log{n}$, $\varepsilon > 0$, and $n$ large enough, we have 
\begin{equation*}
\P\left(\mathrm{Bin}(m, p) = \left\lceil{\frac{\varepsilon}{8} \frac{\log{n}}{\log\round{\log n / mp}}}\right\rceil \right) \geq n^{-\varepsilon}.
\end{equation*}
\end{lemma}
\noindent 
We proceed to estimate $\edmax$, by differentiating the sparse and dense regimes for the average inclusion set size. The following lemma characterizes the two distinct regimes of the maximum inclusion set size $\max_{j \in [n]} |I_j|$. 
\begin{lemma}[Maximum of Binomials] \label{lemma:maximal_ineq}
Let $X_1, \ldots, X_n  \overset{\underset{\mathrm{iid}}{}}{\sim} Bin(m,p)$. Under the conditions in Assumption \ref{asmpt:A_np_m}, it holds that
\begin{align*}
\edmax = \mathbb{E}\max_{i \in [n]}X_i \sim \begin{cases}
        \frac{\log{n}}{\log\round{\log n / mp}} & \text{, if } mp \ll \log{n},\\
        mp & \text{, if } mp \gtrsim \log{n}.
    \end{cases}
\end{align*}
\end{lemma}
\begin{proof}
For ease of notation, let us define $b_n := \frac{\log n}{mp},\: b_n^*:= \frac1e\round{b_n-1}$, $g_n := \frac{\log{n}}{\log\round{\log n / mp}}$. We begin by proving the desired upper bound. By Jensen's inequality and bounding the maximum of positive values by their sum, for any $\lambda > 0$, we obtain
\begin{align*}
    \mathbb{E}\max_{i  \in [n]} X_i & \leq \frac1\lambda \log \mathbb{E}\exp \left(\lambda \max_{i \in [n]} X_i\right) =\frac1\lambda \log \mathbb{E}\left(\max_{i \in [n]} \exp\left(\lambda X_i\right)\right) \leq \frac1\lambda \log \sum_{i \in [n]} \mathbb{E}\exp(\lambda X_i).
\end{align*}
Finally, computing the moment generating function of binomial random variables, together with the inequality $1-x\leq e^{-x}$ yields 
\begin{align*}
    \mathbb{E}\max_{i  \in [n]} X_i = \frac{\log{n} + m\log{(1 - p(1-e^\lambda))}}{\lambda} \leq \frac{\log{n} - mp(1 - e^\lambda)}{\lambda}. 
\end{align*}
In the regime where $mp \gtrsim \log n$, we may choose $\lambda > 0$ arbitrary, independent of $n$, from which it immediately follows that $\E \max_{i  \in [n]} X_i \lesssim mp$. \\
\noindent 
For $mp \ll \log n$, we proceed by differentiating the last line in the above display and setting the resulting expression to zero. From this, we may choose $\lambda$ as the solution of the following.
\[e^{\lambda -1}\round{\lambda -1} = b_n^*\]
Under the present assumptions, this is expressed in terms of the Lambert W function as $\lambda = 1+W_0(b_n^*)$, so that by (\ref{lambert_variational}), we obtain
\begin{align*}
    \E \max_{i \in [n]} X_i \leq \frac{\log n\round{1-\frac{1}{b_n}+\frac{b_n^*}{b_n} \frac{e}{W_0(b_n^*)}}}{1+W_0(b_n^*)} \sim g_n. 
\end{align*}
In the dense $mp \gtrsim \log{n}$ regime, a matching lower bound is easily obtained by noting that $\E \max_{i \in [n]} X_i \geq \E X_1 = mp$. \\
\noindent 
To deal with the sparse regime, let $\tau = 1/16$. From Markov's inequality, 
\begin{align*}\label{markov_initial}
\mathbb{E} \max_{i \in [n]} X_i &\geq  \tau g_n \P\round{\max_{i \in [n]} X_i = \lceil \tau g_n \rceil} = \tau g_n \round{1-\round{1-\P\round{X_1 = \lceil \tau g_n \rceil}}^n}.
\end{align*}
Hence, applying Lemma  \ref{lemma:mp<logn_conquering_logn}, for $n$ large enough, 
\begin{align*}
   \mathbb{E} \max_{i \in [n]} X_i  \geq \tau g_n \round{1-\round{1-n^{-1/2}}^n} \geq (\tau/2) g_n, 
\end{align*}
thus providing a matching lower bound for the sparse regime.
\end{proof}
\begin{remark}
We note that in the sparse regime \(mp \ll \log n\), there holds that \(mp \ll \edmax \ll \log n\), whereas in the dense regime \(mp \gg \log n\), we have that \(mp \sim \edmax \gg \log n\). In the threshold regime \(mp \sim \log n\), the average and maximum of $X_i$'s become of the same order, that is $mp \sim \edmax \sim \log n$. The smooth transition is visible from the proof by noting that in this regime, $b_n, b_n^*, W_0(b_n^*)\sim 1$. 
\end{remark}
\noindent
We conclude this section by bounding $\dmax$ by its expectation from above, up to multiplicative constants w.h.p.. Whilst this one sided result suffices for the forthcoming analysis, we expect a matching lower bound to hold as well. Additional insights into the concentration of $\dmax$ may be found in Lemmas  \ref{lemma:tensorization_var}, \ref{lemma:d_max_concentration}, in Appendix ~A.
\begin{lemma}
\label{lemma:d_max_concentration_whp}
    Let $X_1, \ldots, X_n  \overset{\underset{\mathrm{iid}}{}}{\sim} Bin(m,p)$. Then, there exist constants $c, \tilde{c} > 0$, independent of $n$, such that 
    \[\P \round{\max_{i \in [n]} X_i \geq c \cdot  \E \max_{i \in [n]}X_i} \leq \frac{1}{n^{\tilde{c}}}  \]
    That is, $\max_{i \in [n]}X_i \lesssim \E \max_{i \in [n]}X_i$ w.h.p..
\end{lemma}
\begin{proof}
    Let us consider the sparse and dense regimes separately. \\
    \noindent
    In the dense regime for $mp \gtrsim \log n$, there exist constants $c_1, c_2, c_3 > 0$ such that $c_1 mp \leq \E \max_{i \in [n]} X_i\leq c_2 mp$, as argued in Lemma  \ref{lemma:maximal_ineq}, and $mp\geq c_3\log n$. We apply the union and Chernoff bounds as in Lemma~\ref{lemma:chernoff} to obtain, for any $t \geq 1/c_1$,
    \begin{align*}
        \P \round{\max_{i \in [n]}X_i \geq t \cdot \E \max_{i\in [n]} X_i} &\leq n \P\round{X_1 \geq t c_1 mp} \\
        &\leq n \exp\round{-\frac{(tc_1-1)^2mp}{1+tc_1}} \\
        &\leq  n \exp\round{-\frac{c_3(tc_1-1)^2 \log n}{1+tc_1}}.
    \end{align*}
    It now suffices to choose $t$ as a function of $c_1, c_3$ such that $\frac{c_3(tc_1-1)^2 }{1+tc_1} > 1$. By rearranging and solving the resulting quadratic equation, it follows immediately that any $t > \frac{1}{c_1} + \frac{1+\sqrt{1+8c_3}}{2c_3c_1} > \frac{1}{c_1}$ suffices. Hence, there exist universal constants $c, \tilde{c}$, such that the desired conclusion holds. \\
    \noindent
    We now consider the sparse regime $mp \ll \log n$, where by Lemma  \ref{lemma:maximal_ineq} there exists $c_4 > 0$ such that $mp \leq c_4 \log n / \log \round{\frac{\log n}{\log mp}}$. Notice that for any $\lambda > 0$, $\max_{i \in [n]}X_i \leq \frac{1}{\lambda} \log \sum_{i = 1}^n e^{\lambda X_i}$. We apply Markov's inequality to obtain, for any $t > 0$,
    \begin{align*}
        \P\round{\max_{i \in [n]}X_i\geq t \cdot  \E \max_{i \in [n]}X_i}  & \leq \P\round{\sum_{i = i}^n e^{\lambda X_i} \geq e^{\lambda t \E \max_{i \in [n]}X_i}} \\
        & \leq \frac{n \E e^{\lambda  X_1}}{\exp\round{\lambda  t \: \E \max_{i\in  [n]}X_i}} \\
        & = \frac{n\round{1-p+pe^{\lambda}}^m}{\exp\round{\lambda  t \:  \E \max_{i\in  [n]}X_i}} \\
        & \leq \exp\round{\log n + mp\round{e^{\lambda}-1}-\frac{\lambda t \: c_4 \log n}{\log\round{\frac{\log n}{mp}}}},
    \end{align*}
    where we used that $1 + x < e^x$ to obtain the last inequality. Finally, by choosing $t = 3/c_4$ and $\lambda = \log\round{\log n / mp}$, we obtain
    \begin{align*}
        \P\round{\max_{i \in [n]}X_i\geq \frac{3}{c_4} \cdot  \E \max_{i \in [n]}X_i} \leq \frac{1}{n}.
    \end{align*}
\end{proof}

\section{Algorithmic solutions}\label{sec:algo_sol}

\begin{algorithm}[!t]
\caption{\greedy}\label{alg:greedy}
\begin{algorithmic}[1]
\State \(\mathcal{I} \gets \{I_1, \ldots, I_n\} \) \Comment{Inclusion sets}
\State \(U \gets [m]\)
\State \(t \gets 0\)
\While{\(|U| > 0\)}
    \State \(P \gets \mathrm{argmax}_{I \in \mathcal{I}}\bigl|I \cap U\bigr| \) \Comment{Greedy step}
    \State \(\mathcal{I} \gets \mathcal{I} \setminus \{P\}\)
    \State \(U \gets U \setminus P\)
    \State \(t \gets t + 1\)
\EndWhile
\State \(\valgr \gets t\)
\State \Return \(\valgr\).
\end{algorithmic}
\end{algorithm}

The aim of the present section is to conduct a rigorous analysis of the standard $\greedy$ algorithm for the hitting set problem, within the prescribed Bernoulli random setting. In particular, we show that this routine succeeds at constructing hitting sets of optimal size w.h.p., as in the results of Section~\ref{sec:bounds}, up to multiplicative constants. This is done by first analysing a variation of the greedy heuristic, and subsequently proceeding by a reduction argument.\\
\noindent
The core principle of \greedy is to construct a feasible solution in steps, by sequentially adding to the candidate solution an element which hits the largest number of remaining sets. In the chosen setting, where elements are added to sets with equal probability and independent of each other, we have precise estimates on the number of subsets hit by an element which is {\it picked first}. In fact, the size of this set is given by the maximum of independent Binomial random variables, which was analysed in Section~\ref{sec:d_max}. However, this very first step introduces nontrivial dependencies amongst the remaining matrix columns and significantly complicates keeping track of the marginal gains of each subsequent element addition to the candidate solution.\\
\noindent
In order to circumvent this issue, we introduce a modified greedy routine, which we refer to as \bgreedy\ algorithm, where the elements of the ground set $[n]$ are split into separate sets, which we call blocks. At the \(t\)-th iteration, the algorithm picks the largest column from the first $t$ blocks only. By choosing the number of blocks appropriately, one is guaranteed to find enough independent columns at each iteration, whilst finding a solution of optimal size. Also, the additional block constraint allows us to analyse how many subsets are hit by each chosen element, not just the first one, by still relying on fundamental properties of independent Binomial random variables.\\
\noindent
\bgreedy is detailed in Algorithm~\ref{alg:block_greedy}, whilst informally, it works as follows.
\begin{enumerate}
\item Let \(K\) be the size of the solution (suggested by theoretical analysis);
\item Uniformly at random split \(n\) columns into \(K\) blocks with \(n / K\) columns per block;
\item Start with an empty set of possible choices of columns;
\item At the \(t\)-th iteration, first add the columns
from the \(t\)-th block (Step~\ref{step:adding_new_block}).
Then, perform one greedy step on the current
set of possible choices (Step~\ref{step:greedy_step});
\item If after \(K\) iterations of the algorithm, some subsets remain uncovered, we use  a trivial covering, i.e., covering each subset by a separate column. This can be done with high probability.
\end{enumerate}
Let \(v_t\) be the element which is picked at the \(t\)-th step of \bgreedy, 
\(b_t\) be the number of new subsets that are hit by \(v_t\),
and \(B_t \coloneqq \sum_{i = 1}^t b_i\) be the total number of subsets which are hit after \(t\) steps.
In order to analyse how many elements \bgreedy has picked, we will introduce the sequence 
\(f_1, f_2, \ldots, f_s\), with \(F_t \coloneqq \sum_{i = 1}^t f_i\), such that the following holds:
\begin{enumerate}
    \item for each \(t \leq s\), we have \(F_t \leq B_t\), w.h.p.;
    \item \(F_s = m\);
    \item if \(mp \lesssim \log n\), then \(s \lesssim \vallp\), otherwise, \(s \lesssim \valip\).
\end{enumerate}
The first and second properties ensure that \bgreedy picks at most \(s\) elements, 
and the last property gives optimal bounds on \(s\). 
One way to guarantee that \bgreedy succeeds is to prove that among the choices of \bgreedy at each step \(t\),
there was an element \(\tilde v_t\) which hits at least \(f_t\) new subsets w.h.p. We will prove that 
it is enough to look for \(\tilde v_t\) in 
the new block of columns \(\mathcal{B}_t\), which are added at step \(t\).
Note that unless \(F_t = m\), we have that \(f_t \geq 1\), 
since each subset is hit by at least one element w.h.p.. 
Therefore, it will be enough to find a sequence $\{f_1, f_2, \dots, f_{v}\}$ such that \(F_v \geq m - v\), since it implies \(F_{2v} = m\).
This allows us to reduce the problem of proving the effectiveness of \bgreedy to 
a key technical lemma. This lemma assumes that before step \(t\), exactly \(F_{t - 1}\)
subsets are hit, and bounds from below the probability that some vertex in the new bucket will hit at least \(f_t\) new subsets. 
This boils down to computing \(\P (\mathrm{Bin}(m - F_{t - 1}, p) \geq f_t)\).

\begin{algorithm}[!t]
\caption{\bgreedy}\label{alg:block_greedy}
\begin{algorithmic}[1]
\State Let \(\mathcal{B}_t \subset \curly{I_1,...,I_n}\) denote the \(t\)-th block, i.e. the inclusion sets that {\it become} available at step \(t\).
\State \(\mathcal{I} \gets \varnothing\)
\State $U \gets [m]$
\State $t \gets 0$
\While{\(|U| > 0\) and \(\mathcal{B}_t \neq \varnothing\)}
    \State \(\mathcal{I} \gets \mathcal{I} \cup \mathcal{B}_t\) \label{step:adding_new_block} \Comment{Adding elements from the new block}
    \State \(P \gets \mathrm{argmax}_{I \in \mathcal{I}}\bigl|I \cap U\bigr| \)  \label{step:greedy_step} \Comment{Greedy step}
    \State \(\mathcal{I} \gets \mathcal{I} \setminus \{P\}\)
    \State \(U \gets U \setminus P\)
    \State \(t \gets t + 1\)
\EndWhile
\State \(\valbgr\ \gets t\)
\If{\  \(|U| > 0\) \  } \quad cover the rest of \(U\) with a trivial algorithm, \(\valbgr \gets \valbgr + |U|\)
\EndIf
\State \Return \(\valbgr\).
\end{algorithmic}
\end{algorithm}
\begin{lemma}
\label{lemma:find_one_col_lb}
    Let $\varepsilon > 0$. Consider the following choices of \(f_1, f_2, \ldots\):
    \begin{equation*}
        \begin{aligned}
            &(i) \quad  \text{if } mp \lesssim \log n, \text{ for some constants } \tau > 0 \text{ and } 1 < \alpha < \beta, \\
            &\qquad \qquad f_t = \left\lceil\left(\alpha / \beta \right)^k \tau\mathbb{E}\dmax\right\rceil \qquad \text{where } k \text{ is such that} \quad \beta^{-k-1}m < m - F_{t-1} \leq \beta^{-k}m; \\
            &(ii) \quad  \text{if } mp \gg \log n, \  \text{and }\ \log mp \ll \log n, \\
            &\qquad \qquad f_t =  \ceil{mp(1-p)^{t-1}} \quad \text{if} \quad t \leq t^* \coloneqq \left\lceil \frac{1}{p}\log  \left(\frac{mp}{\log n}\right)\right \rceil,\\
            &\qquad \qquad f_t = \tilde{f}_{t - t^*}, \quad \text{otherwise, where } \tilde{f}_t \text{ is the sequence from the case } mp \lesssim \log n; \\
            &(iii) \quad  \text{otherwise, i.e., when } \log mp \gtrsim \log n, \\
            &\qquad \qquad f_t = \ceil{mp(1-p)^{t - 1}}.
        \end{aligned}
    \end{equation*}

    \noindent
    Then, there exists \(K\), such that 
    \begin{equation}
    \label{eq:cover_size}
    \begin{aligned}
        &(i) \qquad &F_K \geq m - K; \\
        &(ii) \qquad &\text{if } mp \lesssim \log n, &\text{ then } K \sim \vallp; \\
        & \qquad \quad \  &\text{if } mp \gg \log n, &\text{ then } K \sim \valip.
    \end{aligned}
    \end{equation} 
    
    \noindent
    Furthermore, for this sequence \(f_t\) (which depends on \(\varepsilon\)), for any \(t \leq K\),
    \begin{equation}
        \P (\mathrm{Bin}(m - F_{t - 1}, p) \geq f_t) \geq n^{-\varepsilon}.
    \end{equation} 

\noindent
    Note that the implicit constants in the statements \(K \sim \vallp\) or \(K \sim \valip\) depend on \(\varepsilon\).
\end{lemma}
\noindent
This lemma highlights the crucial dependency of the problem 
on the relationship between the average degree, \(mp\), and \(\log n\). We comment on the intuition behind the proof, which can be found in Appendix ~C. As was shown in Section~\ref{sec:d_max}, \(\edmax\) grows identically to the expected value whenever the expected value is large (\(mp \gtrsim \log n\)) and is away from it otherwise. This is the core property for the proof. 

\noindent
When \(mp \sim n^{\gamma}\) for some \(\gamma > 0\), the analysis is straightforward, as just picking random columns is good enough (i.e., picking \(f_t = \ceil{mp(1 - p)^{t - 1}}\)). More challenging cases arise when the average degree is close to \(\log n\), e.g. \(mp = \log^2 n\) (dense) , \(mp = \log n\), or \(mp = \log n / \log \log n\) (sparse). Here we need to carefully track how the maximum degree changes. We look for an element which i) covers a large number of subsets, i.e., close to the expected maximum number, \(\edmax\) and ii) can be found with large enough probability. The second property is important for the reduction to the standard \greedy algorithm, whose direct analysis presents substantial difficulties, and is done later in this section. The quantity $\edmax$ is sensitive to $mp$ whenever the latter is close to $\log{n}$. Hence, we need to adjust which element we look for accordingly. This is done by setting \(f_t = \left\lceil\left(\alpha / \beta \right)^k \tau\mathbb{E}\dmax\right\rceil\) and increasing the parameter \(k\) as the number of remaining rows, \(m - F_t\), decreases. 

\noindent
For example, consider the case \(mp = \log^2 n\). First, we can only pick a random element, since it will be as good as the maximal element. However, during the execution of the algorithm, the problem becomes more sparse, and if we continue to pick random elements, we will construct a suboptimal solution. Therefore, once we are in the regime \(mp \sim \log n\), we start to gradually increase how much the newly picked element will cover, with respect to a random element.

\noindent
It is now straightforward to prove the following theorem, which makes rigorous the statements in Section~\ref{sec:intro}.

\begin{theorem}
    \label{thm:alg_solution_ub}
        Under Assumption~\ref{asmpt:A_np_m}, we have that
        \begin{equation}
            \begin{aligned}
                &(i) \text{ if } \  mp \lesssim \log n \  \text{then, for any } \varepsilon >0\text{ and } n \text{ large enough},\\ 
                &\qquad \P\left(\valbgr \lesssim \frac{m}{\mathbb{E}\dmax}\right) \geq 1 - \exp\left(-n^{1 - \delta - \varepsilon}\right); \\
                &(ii) \text{ if } \ mp \gg \log n, \text{ then, for any } \varepsilon > 0 \text{ and } n \text{ large enough}, \\
                &\qquad\P\left(\valbgr \lesssim \frac{1}{p} \log\left(\frac{mp}{\log n}\right)\right) \geq 1 - \exp\left(-n^{1 - \delta - \varepsilon}\right). \\
            \end{aligned}
        \end{equation}
        Note that if \(mp \gtrsim n^{\gamma}\) for some \(\gamma > 0\), then \(\log \frac{mp}{\log n} \sim \log n\), and the bound in \((ii)\) can be simplified.
\end{theorem}
\begin{proof}
The main idea of the proof is to analyse the distribution of the columns that are added at each step \(t\). These columns are independent, and for each newly added column, the number of additional subsets which it covers is distributed according to \(\mathrm{Bin}(m - F_{t - 1}, p)\), where \(F_{t - 1}\) is the number of subsets which are already covered. Lemma~\ref{lemma:find_one_col_lb} allows us to lower bound \(F_t\), and we show now that we can do this with high probability.

\noindent
Fix \(\varepsilon > 0\) and let \(\varepsilon' \coloneqq \varepsilon / 4\). Let \(f_1, f_2, \ldots\) be the sequence from Lemma~\ref{lemma:find_one_col_lb} for \(\varepsilon'\) and \(K\) be the value for which (\ref{eq:cover_size}) is satisfied, i.e. \(F_{K} \geq m - K\). Notice that $K \leq C \max\left\{ \frac{m}{\E \dmax}, \frac{1}{p} \log(\frac{mp}{\log{n}}) \right\}$ for some constant $C > 0$, for \(n\) large enough. We uniformly at random split \(n\) elements (columns) into \(K\) groups of size \(n / K\) each (assuming without loss of generality that $K$ divides $n$, otherwise we consider groups of size \(\floor{n / K}\)), so that $\mathcal{B}_t$ yields a new set of $n/K$ elements at each iteration $t \leq K$ and \(\mathcal{B}_t = \varnothing\) for \(t > K\). We say that the algorithm fails at step \(t\) if before step \(t\), at least \(m - F_{t - 1}\) subsets are covered, but after step \(t\) less than \(m - F_t\) sets are covered. 
Using that, for \(n\) large enough, $(i)$ columns in each newly added block are independent, \((ii)\) \(\P \left(\mathrm{Bin}(m - F_{t - 1}, p) \geq f_t\right) \geq n^{-\varepsilon'}\), and \((iii)\) \(n / K \geq n^{1 - \delta - \varepsilon'}\), we get
\begin{align*}
\P\left(\text{\bgreedy fails at step t} \right) &\overset{(i)}\leq \left(\P\left(\mathrm{Bin}(m - F_{t - 1}, p) < f_t\right) \right)^{n / K} \\
&\overset{(ii)}{\leq} \left( 1 - n^{-\varepsilon'} \right)^{n / K} \\
& \overset{(iii)}{\leq} \exp\left( - n^{1 - \delta - 2\varepsilon'}\right).
\end{align*}
We then proceed by applying a union bound to obtain the result,
\begin{align*}
\P \left(\bgreedy\text{ fails during first } K \text{ steps}\right) &\leq \sum_{t=1}^K \P\left( \bgreedy\text{ fails at step t}\right) \\
&\leq K \cdot \exp\left(- n^{1 - \delta - 2\varepsilon'}\right) \\
&\leq \exp\left(- n^{1 - \delta - 3\varepsilon'} \right),
\end{align*}
where the second inequality holds since, by definition, the algorithm runs for $K$ iterations, and the third one holds for $n$ large enough. 
We proved that \bgreedy succeeds in finding at most $K$ elements such that at most $m - F_{K}$ sets remain uncovered. Since by construction, $m - F_{K} \leq K$, we can cover the remaining rows trivially using that \(\ip\) is feasible by Lemma~\ref{lemma:lp_ub} with high probability, which proves that 
\begin{equation*}
    \P \left(\valbgr \leq 2K\right) \geq 1 - \exp\left(-n^{1 - \delta - 4\varepsilon'}\right) =1 - \exp\left(-n^{1 - \delta - \varepsilon}\right) ,
\end{equation*}
for $n$ large enough. Recalling that \(K \lesssim \vallp\) for \(mp \lesssim \log n\), and that \(K \lesssim \valip\) for \(mp \gg \log n\), finishes the proof. 
\end{proof}
\noindent
With the above results at hand, we now proceed to analyse the \(\greedy\) algorithm by means of a suitable reduction. Recall that we denote outputs of \(\bgreedy\) and \(\greedy\) as \(\valbgr\) and \(\valgr\) respectively. 
\begin{theorem} \label{thm:greedy}
Under Assumption \ref{asmpt:A_np_m} with $\delta < 1/2$, we have that, for \(n\) large enough,
\begin{equation*}
    \P\left(\valgr \sim \valip \right) \geq 1 - \exp\left(-\sqrt{n}\right).
\end{equation*}
\end{theorem}
\begin{proof}
We use Theorem~\ref{thm:alg_solution_ub} with \(\varepsilon = 1/8 - \delta / 4\), and let $K, \mathcal{B}_t$ be as defined in the proof of Theorem~\ref{thm:alg_solution_ub}. 
We have that, for \(n\) large enough,
\begin{align*}
\P \left(\bgreedy\text{ fails at any step}\right) \leq \exp\left(-n^{\Delta} \right),
\end{align*}
where \(\Delta \coloneqq 3/4 - \delta / 2 > 1/2\). 

\noindent
Given a matrix $\A$, consider running the above definition of \bgreedy for $J := \exp(\sqrt{n})$ times, each time reshuffling the columns. Both \(\valbgr\) and \(\valgr\) are random variables, but conditioned on $\A$, \(\valgr\) is deterministic, while \(\valbgr\) still depends on the randomness of separating columns into blocks.
Using the union bound, we have that
\begin{equation}
\begin{aligned}
\label{eq:gr_first_ub}
\P\round{\valgr > 2K } \leq 
\ &\P\round{\exists\text{ a failed copy of }\bgreedy} \\
 &+ \P\round{\valbgr < \valgr \text{ over all $J$ copies} }.
\end{aligned}
\end{equation}
Applying the union bound again, we can upper bound the first term in~(\ref{eq:gr_first_ub}):
\begin{equation}
\label{eq:gr_second_ub}
    \P\left(\exists\text{ a failed copy of \bgreedy}\right) \leq J \exp\left(- n^{\Delta} \right) = \exp\left(-n^{\Delta} + n^{1/2} \right).
\end{equation}
Now we focus on the second term in~(\ref{eq:gr_first_ub}).
Let \(v_1, v_2, \ldots, v_g\) be the ordered sequence of elements picked by \(\greedy\). 
Let \(M_t \coloneqq \{v_1 \in \mathcal{B}_1, v_2 \in \mathcal{B}_1 \cup \mathcal{B}_2, \ldots, v_{t} \in  \mathcal{B}_1 \cup \ldots \cup \mathcal{B}_t\}\).
The event \(\{\valbgr \geq \valgr\}\) contains the event \(M_g\), since in this case \(\bgreedy\) will necessarily pick exactly the same columns \(v_1, v_2, \ldots, v_g\). Given that each reshuffling of the columns generates a uniform distribution of \(\mathcal{B}_i\)'s over possible partitions of \(n\) columns, we get that
\begin{align*}
\P\left(M_g\right) = \P\left(v_1 \in \mathcal{B}_1 \right)\P\left(v_2 \in \mathcal{B}_1 \cup \mathcal{B}_2 \mid M_1 \right) \ldots \P\left(\v_g \in \mathcal{B}_1 \cup \ldots \cup \mathcal{B}_g \mid M_{g - 1} \right).
\end{align*}
The \(t\)-th term in the product above is equal to  \begin{equation*}
    \P(\v_t \in \mathcal{B}_1 \cup \ldots \cup \mathcal{B}_t \mid M_{t - 1}) =  \frac{t \frac{n}{K} - (t - 1)}{n - (t - 1)} \geq \frac{t}{K} - \frac{t - 1}{n} \geq \frac{t}{2(K - 1)},
\end{equation*}
where the last inequality holds for \(n \geq 4K\) (recall that \(n \gg K\)). Since \(M_g \subset \{\valbgr \geq \valgr\}\), we can lower bound the probability of the latter event as follows (note that when \(g < K\) there will be less terms in the product, hence, \(\P(M_g)\) will be even larger),
\begin{align*}
\P\left(\valbgr \geq  \valgr \text{ for } 1 \text{ copy}\right) \geq \P(M_g) \geq \prod_{t=1}^{K - 1} \P(\v_t \in \mathcal{B}_1 \cup \ldots \cup \mathcal{B}_t \mid M_{t - 1})  \geq \prod_{t=1}^{K - 1} \frac{t}{2(K - 1)} \geq e^{-2K},
\end{align*}
where we used that \(k! \geq (k/e)^k\) in the last inequality.
Since \(K \leq C \max\left\{ \frac{m}{\E \dmax}, \frac{1}{p} \log(\frac{mp}{\log{n}}) \right\}\) and \(1/p \leq n^{\delta}\), there exists a constant \(\tilde C > 0\) large enough, such that \(K  \leq \tilde C n^{\delta} \log n\). Therefore, using independence of the reshuffling between the copies, we can compute
\begin{equation}
\label{eq:gr_third_ub}
\begin{aligned}
\P\left(\valbgr < \valgr \text{ over all $J$ copies}   \right)& = \left( 1 - \P\left(\valbgr \geq  \valgr \text{ for } 1 \text{ copy} \right) \right)^J \\
& \leq (1 - e^{-2K})^J \\
& \leq \exp\left(- e^{\sqrt{n}- 2\tilde C n^{\delta} \log{n}} \right).
\end{aligned}
\end{equation}
Combining (\ref{eq:gr_first_ub}), (\ref{eq:gr_second_ub}) and (\ref{eq:gr_third_ub}), we showed that $\P\left(\valgr > 2K \right) \leq \exp\left(-\sqrt{n}\right)$ for \(n\) large enough, which finishes the proof.  
\end{proof}

\begin{remark}
We note that the $\delta < 1/2$ condition in Theorem~\ref{thm:greedy} is likely not optimal, and could be relaxed by reducing to \bgreedy with more carefully chosen sets $\mathcal{B}_t$. In particular, the appropriate set sizes $|\mathcal{B}_t|$ may not be identical across $t \leq K$. The analysis becomes more technical in this case, and we highlight this as an interesting open direction.
\end{remark}




\section{Discussion and Open Questions}\label{sec:discussion}


            
            

Our work characterises multiplicative integrality gaps for the random hitting set problem. In this section, we discuss the intuition behind our main results, together with open questions and conjectures.
\subsection{Summary of our results and proof techniques.}
We identified that the nature of integrality gaps depends on the size of the inclusion set, also viewed as the sparsity of the underlying hypergraph. In particular, when the average degree of a vertex is small, i.e., when each element belongs to a small number of subsets, we proved that there exists only a constant gap between linear and integer program solutions, together with a simple algorithmic solution. The situation changes when the hypergraph becomes dense, where we show an increasing integrality gap. This separation stems mostly from the property of the binomial distribution, where the maximum of random variables grows identically to the expected value whenever the expected value is large, but is away from it if \(mp \ll \log n\).\\
\noindent
 In our analysis of \bgreedy, we track this change of behaviour using a geometric series, which means that the further we are in the execution of the algorithm, the larger the ratio between the element we pick and the average element will be. This picture coincides exactly with how the binomial distribution will behave if we decrease the average degree: for large instances, it will look approximately as a Gaussian, but when the average degree is small, Poisson approximation starts to dominate, the right tail becomes heavier, and the difference between \(\dmax\) and \(mp\) increases. Our analysis tracks the transition between Gaussian and Poisson-like behavior.
\subsection{Multiplicative vs. additive integrality gaps}
Our result only concerns multiplicative gaps, but the constants in our analysis can be large. This might be a consequence of the generality of the studied problem. For example, if one focuses only on the case of constant \(p\), which immediately implies a very dense instance in our characterization, \cite{telelis2005absolute} proves that a simple algorithm is optimal for approximating the integer program up to a small additive error. Proving similar upper bounds on the constant in more general cases is an interesting open problem. Based on numerical experiments, we formulate the following conjectures.

\begin{conjecture}[Very sparse]
For \(mp \ll 1\),
\begin{equation}
    \frac{\valgr}{\vallp} \to 1.
\end{equation}
\end{conjecture}
\begin{conjecture}[Sparse]
For \(1 \lesssim mp \ll \log n\),
\begin{equation}
    \frac{\valgr}{\valip} \to 1, \qquad \qquad \frac{\valip}{\vallp} \to C_1,
\end{equation}
where \(1 < C_1 < 1.5\).
\end{conjecture}
\begin{conjecture}[Dense]
    For \(mp \gg \log n\),
    \begin{equation}
        \frac{\valgr}{\valip} \to C_2,
    \end{equation}
    where \(1 \leq C_2 < 1.5\).
\end{conjecture}
\subsection{Analysis of a linear program solution.}
One motivation for studying the gaps between the integer and linear programs together with the solutions of linear programs themselves is to construct a rounding scheme which converts a fractional solution to an integer one. We believe this is another interesting direction for future work. In particular, numerical experiments show that entries which have large value in the fractional solution have a strong tendency to correspond to elements that are picked for the integer solution. This supports the claim that a combination of the greedy and linear programming approach might be fruitful in efficiently solving $\hs$. One approach for further study consists of first solving a linear program, initializing $\x$ with the largest elements in the linear solution, and greedily covering the remaining subsets.

\section{Acknowledgments}
\noindent
The authors thank Dylan J. Altschuler, Afonso S. Bandeira, Raphaël Barboni, and Anastasia Kireeva for helpful discussions. DD is supported by ETH AI Center doctoral fellowship and ETH Foundations of Data Science initiative. GA is supported by the Cambridge Trust and Invenia Labs. NG is grateful for the funding received from Elizaveta Rebrova.

\bibliography{references.bib}
\bibliographystyle{abbrv}

\newpage
\section*{Appendix A. Auxiliary lemmas}\label{appendix:aux_lemmas}
\begin{lemma}\label{lemma:chernoff}(Chernoff Bound - upper tail)
Let $X_1, ..., X_n$ be independent random variables taking values in $\curly{0, 1}$, $X$ denote their sum and $\mu = \E X$. Then for any $\delta > 0$,
    \begin{align*}
        \P\round{X \geq (1+\delta)\mu}\leq e^{-\delta^2 \mu / (2+\delta)}.
    \end{align*}
\end{lemma}

\noindent
In order to deal with concentration of $\dmax$ around its expectation, we state the following useful result on tensorization of variance. We introduce notation \(\text{Var}_i\) and \(\E_i\), where subscript $i$ indicates conditioning on each component of an underlying random vector, except for the $i$-th one.
\begin{lemma}[Theorem 2.3, \cite{van2014probability}]\label{lemma:tensorization_var}
Let $X_1,...,X_n$ be independent random variables and for each function $f:\mathbb{R}^n \rightarrow \mathbb{R}$, define 
\[\text{Var}_i f(x_1,...,x_n) := \text{Var}\left(x_1,...,x_{i-1},X_i,x_{i+1},...,x_n\right).\]
Then, there holds that 
\[\text{Var} \left(f\round{X_1,...,X_n}\right ) \leq \E \sum_{i = 1}^n \text{Var}_i f\round{X_1,...,X_n}\]
\end{lemma}

\begin{lemma}[Concentration for $\dmax$]\label{lemma:d_max_concentration}
    Let $X_1, \ldots, X_n  \overset{\underset{\mathrm{iid}}{}}{\sim} Bin(m,p)$. Then, for any $t > 0$, 
    \begin{align*}
        \P\round{\left |\dmax-\E \dmax\right|> t}\leq \frac{mp}{t^2}.
    \end{align*}
\end{lemma}
\begin{remark}
    Note that in all regimes of $m, p$ satisfying Assumption \ref{asmpt:A_np_m}, choosing $t \sim {\E \dmax}$ is sufficient to deduce from the previous lemma that $\dmax\sim \edmax$ w.h.p..
\end{remark}
\begin{proof}
    Proceeding by Chebyschev's inequality, it suffices to show that $\text{Var}(\dmax\leq mp$. By Lemma~\ref{lemma:tensorization_var}, we have that 
    \begin{align*}
        \text{Var}(\dmax) & \leq \E \sum_{i = 1}^n \E_i\round{ \dmax - \E_i \dmax}^2\\
        & = \E \sum_{i = 1}^n \E_i\left[\round{ \dmax - \E_i \dmax}^2\mid \dmax = X_i \right]\P{\dmax = X_i} \\
        & +  \E \sum_{i = 1}^n \E_i\left[\round{ \dmax - \E_i \dmax}^2\mid \dmax \neq X_i \right]\P{\dmax \neq X_i} \\
        & = \frac1n\E \sum_{i = 1}^n \text{Var}X_i \\
        & \leq mp,
    \end{align*}
which is as required.
\end{proof}

\begin{lemma}[Asymptotic expression for binomial probability mass function]
\label{lemma:binom_lb}
~\\~
    Let \(a \equiv a(n)\) and \(b \equiv b(n)\) be such that
    \begin{enumerate}
        \item \(1 \ll b \ll \sqrt{a}\),
        \item \(p \ll 1\). 
    \end{enumerate}
    If \(b \geq C ap\) for \(C > 1\), then 
    \begin{equation}
    \label{eq:binom_lb_small_b}
        \log \P(\mathrm{Bin}(\ceil a, p) = \ceil b) \geq -\left(b \log \frac{b}{a p} - b + ap\right) (1 + o(1)),
    \end{equation}
    If also \(b \gg ap\), we have that 
    \begin{equation}
    \label{eq:binom_lb_large_b}
        \log \P(\mathrm{Bin}(\ceil a, p) = \ceil b) \geq -\left(b \log \frac{b}{a p}\right) (1 + o(1)),
    \end{equation}
Furthermore, all bounds remain valid upon replacing \(\ceil a\) to \(\floor a\).
\end{lemma}
\begin{proof}
    \begin{align*}
\P(\mathrm{Bin}(\ceil a, p) = \lceil b \rceil) &= {\ceil a \choose \lceil b \rceil} p^{\lceil b \rceil}(1 - p)^{\ceil a - \lceil b \rceil} \\
&\overset{(i)}{\geq} \frac{(\ceil a p)^{  \ceil{b}  }}{4 (  \ceil{b}  )!} (1 - p)^{\ceil a -   \ceil{b}  } \\
&\overset{(ii)}{\geq} \frac{1}{\ceil b} \left(\frac{\ceil a e p}{  \ceil{b}  } \right)^{  \ceil{b}  } (1 - p)^{\ceil a -   \ceil{b}  } \\
&\overset{(iii)}{\geq} \frac{1}{\ceil b} \left(\frac{\ceil a e p}{  \ceil{b}  } \right)^{  \ceil{b}  } e^{-\ceil a p} (1 - p)^{\ceil a p - \ceil{b} },
\end{align*}
where $(i)$ is due to ${n \choose k} \geq \frac{n^k}{4k!}$ for $0 \leq k \leq \sqrt{n}$, $(ii)$ is due to $n! \leq \frac{n}{4}(n/e)^n$ for $n$ large enough, and $(iii)$ is due to $(1 + x/n)^n \geq e^x (1 - x^2/n)$ for $|x|\leq n$.
After taking the logarithm, we get 
\begin{equation*}
    \log \P(\mathrm{Bin}(\ceil a, p) = \lceil b \rceil) \geq - \left(\ceil b \log \frac{\ceil b}{\ceil a  p} - \ceil b + \ceil a p +  \log \ceil{b} - (\ceil a p - \ceil b) \log (1 - p) \right).
\end{equation*}
If \(b \geq C \ceil a p\) for \(C > 1\), we have that 
\begin{equation*}
    b \log \frac{b}{ap} - b + ap \geq \gamma b \gg 1, \qquad \text{for } \gamma := \frac{1}{C} + \log C - 1 > 0.
\end{equation*}
Since \(b \gg 1\), we have
\begin{equation*}
\frac{\ceil b \log \frac{\ceil b}{\ceil a p} - \ceil b + \ceil a p}{ b \log \frac{b}{ap} - b + ap} = 1 + o(1).
\end{equation*}
Now, since also \(\log \ceil b \ll b\) and \((\ceil a p - \ceil b) \log (1 - p) \ll b\) for \(p \ll 1\), we have that
\begin{equation*}
    \log \P(\mathrm{Bin}(\ceil a, p) = \ceil b) \geq - \left(b \log \frac{b}{ap} - b + ap\right)(1 + o(1)).
\end{equation*}
If additionally \(b \gg ap\), then 
\begin{equation*}
    \frac{b \log \frac{b}{ap} - b + ap}{b \log \frac{b}{ap}} = 1 + o(1),
\end{equation*}
and, finally,
\begin{equation*}
    \log \P(\mathrm{Bin}(\ceil a, p) = \lceil b \rceil) \geq - \left(b \log \frac{b}{ap}\right)(1 + o(1)).
\end{equation*}
Under our assumptions, \(1 \ll b \ll \sqrt{a}\), the same bounds hold for \(\log \P(\mathrm{Bin}(\floor a, p) = \floor b)\).
\end{proof}

\begin{lemma} [Binomial Monotonicity] \label{lemma:binomial_monotonicity}
Let $S_m \sim \mathrm{Bin}(m, p)$. Then for $r \geq mp$, we have that $\P(S_m = r + 1) \leq \P(S_m = r)$ and \(\P(S_{m-1} = r) \leq \P(S_{m} = r)\).
\end{lemma}
\begin{proof}
The proof follows a similar argument as that presented in \cite{bunke_feller_1969}.  
\begin{align*}
\frac{\P(S_m = r + 1)}{\P(S_m = r)} &= \frac{ {m \choose r+1} p^{r+1} (1 - p)^{m - r - 1} }{ {m \choose r} p^r (1 - p)^{m - r} } \\
&= \frac{ \frac{m!}{(r + 1)! (m - r - 1)!} p^{r + 1} (1-p)^{m - r - 1}}{ \frac{m!}{r! (m - r)!} p^r (1-p)^{m - r}} \\
&= \frac{(m - r) p }{(r + 1) (1-p)} \leq 1. 
\end{align*}
Similar arguments show that \(\P(S_{m-1} = r) \leq \P(S_{m} = r)\).
\end{proof}
\vspace{1em}
%
%
\section*{Appendix B. Main tool for the case \(mp \lesssim \log n\) and Proof of Lemma~\ref{lemma:mp<logn_conquering_logn}}\label{app:geom_alg}

\begin{lemma}
\label{lemma:geom_alg}
    If \(mp \lesssim \log n\), then, for any \(\varepsilon > 0\), there exist constants \(\tau > 0\) and \(1 < \alpha < \beta\), such that, for \(k \lesssim \log n\) and for any \(\tilde m\), satisfying \(\beta^{-k-1}m \leq \tilde m \leq \beta^{-k} m\), for all \(n\) large enough,
    \begin{equation*}
        \P\biggl(\mathrm{Bin}(\tilde m, p) = \bigl\lceil (\alpha / \beta)^{k} \tau \edmax\bigr\rceil \biggr) \geq n^{-\varepsilon}.
    \end{equation*}
\end{lemma}
\begin{proof}
    The proof is essentially a careful application of Lemma~\ref{lemma:binom_lb}.
    Let \(\tau, \alpha, \beta\) be constants to be fixed later and \(\tilde m = \floor{\beta^{-k-1}m}\). Depending on whether we have \(mp \ll \log n\) or \(mp \sim \log n\), different terms will dominate the asymptotic expression from Lemma~\ref{lemma:binom_lb}.

    \noindent
    We start with the case \(mp \ll \log n\). From Lemma~\ref{lemma:maximal_ineq}, this implies that \(mp \ll \edmax \ll \log n\). Here we can fix \(\alpha \equiv 2\) and \(\beta \equiv 3\). Applying~(\ref{eq:binom_lb_large_b}) for \(a = 3^{-k-1}m\) and \(b = (2/3)^k \tau \edmax\), we have:
    \begin{equation}
        \log \P\biggl(\mathrm{Bin}(\tilde m, p) = \bigl\lceil (2/3)^k \tau \edmax\bigr\rceil \biggr) \geq - (2/3)^k \tau \edmax \log \left(\frac{2^k 3 \tau \edmax}{ mp}\right) (1 + o(1))
    \end{equation}
    Recall that our goal is to show \(\log \P\biggl(\mathrm{Bin}(\tilde m, p) = \bigl\lceil (2/3)^k \tau \edmax\bigr\rceil \biggr) \geq -\varepsilon \log n\). We first show that there exists \(\tau > 0\) satisfying the following two inequalities:
    \begin{equation}
    \begin{aligned}
        &(i) \qquad (2/3)^k \tau (\log 3 + k \log 2 ) \frac{\edmax}{\log n} &\leq \quad \frac{\varepsilon}{4}, \\
        &(ii) \qquad (2/3)^k \tau \frac{\edmax}{\log n} \log \left(\frac{\edmax}{mp}\right) &\leq \quad \frac{\varepsilon}{4}.
    \end{aligned}
    \end{equation}
    Indeed, since \(\edmax \ll \log n\) and \(k \ll (3/2)^k\), inequality \((i)\) will be satisfied for any \(\tau > 0\) for \(n\) large enough. For \((ii)\) we need to use explicit bound for \(\edmax\), in particular from Lemma~\ref{lemma:maximal_ineq} we know that there exists \(C > 0\), such that \(\edmax \leq C \log n / (\log \log n - \log mp)\) for \(n\) large enough. Plugging this into \((ii)\), we get for \(k = 0\),
    \begin{equation}
        \tau \frac{\mathbb{E}\dmax}{\log n}\log\left(\frac{\mathbb{E}\dmax}{mp}\right) \leq \frac{\tau C (\log C + \log \log n - \log (\log \log n - \log mp) - \log mp)}{\log \log n - \log mp} = \tau C + o(1).
    \end{equation}
    For \(\tau = \varepsilon / (8C)\), \((ii)\) holds for \(k = 0\) for \(n\) large enough. By increasing \(k\) we only decrease left hand side of \((ii)\), therefore, the same value of \(\tau\) works for any \(k \geq 0\). 

    \noindent
    Finally, by adding \((i)\) and \((ii)\) we showed that, for \(n\) large enough, 
    \begin{equation*}
    \log \P\biggl(\mathrm{Bin}(\tilde m, p) = \bigl\lceil (\alpha/2)^k \tau \edmax\bigr\rceil \biggr) \geq - \frac{\varepsilon}{2}\log n (1 + o(1)) > -\varepsilon \log n,
    \end{equation*}
    which finishes the proof for the case \(mp \ll \log n\).

    \noindent
    Now we focus on the case \(mp \sim \log n\). Here we apply~(\ref{eq:binom_lb_small_b}) for the values \(a = {\beta}^{-k-1}m\) and \(b = (\alpha / \beta)^k \tau \edmax\) keeping in mind the condition \(b \geq Cap\) with \(C > 1\). 
    We have
    \begin{equation*}
    \begin{aligned}
         &\log \P\biggl(\mathrm{Bin}(\tilde m, p) = \bigl\lceil (\alpha / \beta)^k \tau \edmax\bigr\rceil \biggr) \\
         &\quad \geq - \left((\alpha / \beta)^k \tau \edmax \log \left(\frac{\beta \alpha^k \tau \edmax}{ mp}\right) - (\alpha / \beta)^k \tau \edmax + \beta^{-k-1}mp\right)(1 + o(1)) 
    \end{aligned}
    \end{equation*}
    We pick \(\tau = \gamma mp / \edmax\), for some constant \(\gamma > 1\) to be specified later. Note that this way condition for applying~(\ref{eq:binom_lb_small_b}), \(\frac{b}{ap} \geq C > 1\), is satisfied since \(\frac{b}{ap} \geq \frac{\tau \edmax}{mp} = \gamma > 1\). This simplifies the latter expression to the following:
       \begin{equation*}
    \begin{aligned}
         &\log \P\biggl(\mathrm{Bin}(\tilde m, p) = \bigl\lceil (\alpha / \beta)^k \gamma mp\bigr\rceil \biggr) \\
         &\quad \geq - mp\left((\alpha / \beta)^k \gamma \log \bigl(\beta \gamma \alpha^k\bigr) - (\alpha / \beta)^k \gamma + \beta^{-k-1}\right)(1 + o(1)) 
    \end{aligned}
    \end{equation*}
Since in this regime we have \(mp \leq D \log n\) for some \(D > 0\), for \(n\) large enough, it is enough to show 
\begin{equation*}
   (\alpha / \beta)^k \gamma \log \bigl(\beta \gamma \alpha^k\bigr) - (\alpha / \beta)^k \gamma + \beta^{-k-1} \leq \varepsilon / (2D).
\end{equation*}
We first show that there exist constants \(1 < \alpha < \beta\) and \(\gamma > 1\), depending on \(\varepsilon\) and \(D\), satisfying the following two inequalities for any \(k \geq 0\):
\begin{equation*}
    \begin{aligned}
        &(i) \qquad (\alpha/\beta)^k \left(\gamma \log \beta \gamma - \gamma + \frac{1}{\alpha^k \beta}\right) &\leq \quad \frac{\varepsilon}{4D}, \\
        &(ii) \qquad (\alpha/\beta)^k k \log \alpha & \leq \quad \frac{\varepsilon}{4D}.
    \end{aligned}
\end{equation*}
Note that left hand side of \((i)\) decreases as \(k\) increases, therefore, it is enough to look at \(k = 0\). We need to show that there exist \(\beta, \gamma > 1\), depending on \(\varepsilon, D\) such that
\begin{equation*}
    f(\beta, \gamma) := \gamma \log \beta \gamma - \gamma + \frac{1}{\beta} \leq \frac{\varepsilon}{4D}.
\end{equation*}
Note that \(\frac{\partial f}{\partial \beta} = \gamma / \beta - 1 / \beta^2 > 0\) and \(\frac{\partial f}{\partial \gamma} = \log \beta \gamma > 0\) as long as \(\beta \gamma > 1\). Since \(f(1, 1) = 0\), we can find \(\beta, \gamma > 1\), close enough to 1, such that \(f(\beta, \gamma) \leq \varepsilon / (4D)\).
We use these values of \(\beta\) and \(\gamma\) (or, equivalently, \(\tau\)). Since \(k \ll (\beta / \alpha)^k\), there exists \(\alpha \in (1, \beta)\), such that \((ii)\) holds. Summing \((i)\) and \((ii)\) shows that, for \(n\) large enough,
\begin{equation*}
    \log \P\biggl(\mathrm{Bin}(\tilde m, p) = \bigl\lceil (\alpha / \beta)^k \gamma mp\bigr\rceil \biggr) \geq - \frac{\varepsilon mp}{2D}(1 + o(1)) \geq - \frac{\varepsilon \log n}{2} (1+o(1)) \geq -\varepsilon \log n.
\end{equation*}
    We proved that for \(mp \lesssim \log n\), for any \(\varepsilon > 0\), for \(n\) large enough, there exists \(\tau, \alpha, \beta\), such that
       \begin{equation*}
        \mathrm{Pr}\biggl(\mathrm{Bin}(\floor {\beta^{-k-1} m}, p) = \lceil(\alpha/\beta)^k \tau \mathbb{E}\dmax\rceil \biggr) \geq n^{-\varepsilon}.
    \end{equation*}
    Since \(\beta^{-k-1}mp < \beta^{-k}mp < \lceil(\alpha/\beta)^k \tau \mathbb{E}\dmax\rceil\), from binomial monotonicity, Lemma~\ref{lemma:binomial_monotonicity}, we have that for any \(\tilde m\) such that \(\beta^{-k - 1} m \leq \tilde m \leq \beta^{-k} m\),
    \begin{equation*}
    \P\Bigl(\mathrm{Bin}(\tilde m, p) = \lceil(\alpha / \beta)^k \tau \mathbb{E}\dmax\rceil \Bigr) \geq n^{-\varepsilon}.
    \end{equation*}

\end{proof}
\begin{proof}[Proof of Lemma~\ref{lemma:mp<logn_conquering_logn}]
We follow the argument in Lemma~\ref{lemma:geom_alg} with \(k = 0\) and \(\edmax\) replaced by \(\log n / (\log \log n - \log mp)\). Note that in the proof of Lemma~\ref{lemma:geom_alg}, in the case \(mp \ll \log n\), we only used that \(mp \ll \edmax \ll \log n\) and \(\edmax \leq C \log n / (\log \log n - \log mp)\) for some \(C > 0\). Since both these properties remain true upon replacing \(\edmax\) with \(\log n / (\log \log n - \log mp)\), the proof follows. Since \(\tau = \varepsilon / (8C)\), in the setting of Lemma~\ref{lemma:mp<logn_conquering_logn}, and \(C = 1\) in this argument, we pick \(\tau = \varepsilon / 8\).
\end{proof}
\section*{Appendix C. Proof of lemma \ref{lemma:find_one_col_lb}} \label{sec:appendix_proof_lemma_8}
We proceed in the proof by first showing that there exists \(\tilde K\), such that \(m - F_{\tilde K} \lesssim \tilde K\), and then, by increasing \(\tilde K\) by a multiplicative factor, we find \(K\) such that \(m - F_K \leq K\).
\subsection*{Case \(mp \lesssim \log n\)} 
From Lemma~\ref{lemma:geom_alg}, there exist constants \(\tau > 0\), \(\alpha, \beta\) with \(1 < \alpha < \beta\), such that, for any \(\tilde m\), satisfying \(\beta^{-k-1}m \leq \tilde m \leq \beta^{-k} m\), for all \(n\) large enough,
    \begin{equation*}
        \P\biggl(\mathrm{Bin}(\tilde m, p) = \bigl\lceil (\alpha / \beta)^{k} \tau \edmax\bigr\rceil \biggr) \geq n^{-\varepsilon}.
    \end{equation*}
Recall that in this case \(f_t  = \bigl\lceil (\alpha / \beta)^{k} \tau \edmax\bigr\rceil\), where \(k\) is such that \(\beta^{-k-1}m \leq m - F_{t-1} \leq \beta^{-k} m\) and \(F_t = \sum_{s = 1}^t f_s\). From Lemma~\ref{lemma:geom_alg} we have that \(\P(\mathrm{Bin}(m - F_{t-1}, p) = f_t) \geq n^{-\varepsilon}\). Our goal is to prove that there exists \(s \lesssim \vallp \sim m / \edmax\), such that \(m - F_s \lesssim s\). 

\begin{lemma} \label{lemma:one_step_beta}
    Let \(t^{(k)} := \frac{\beta - 1}{\beta\tau} \frac{m}{\edmax} \alpha^{-k}\). 
    \begin{equation*}
        \begin{aligned}
            \text{If} \qquad &m - F_{t-1} \leq \beta^{-k} m  \\
            \text{then} \qquad &m - F_{t+t^{(k)}-1} \leq \beta^{-k-1} m.
        \end{aligned}
    \end{equation*}
    Informally, if after \(t - 1\) steps of \bgreedy, at most \(\beta^{-k} m\) subsets are uncovered, then after \(t + t^{(k)} - 1\) steps, at most \(\beta^{-k-1}m\) subsets remain uncovered.
\end{lemma}
\begin{proof}
    Let \(s \geq t\). As long as \(m - F_{s - 1} > \beta^{-k-1}m\), we will always have \(f_s = \bigl\lceil (\alpha / \beta)^{k} \tau \edmax\bigr\rceil\). We proceed by contradiction. Assume that \(m - F_{t + t^{(k)} - 1} > \beta^{-k-1}m\). This implies that for all \(s \in [t - 1, t + t^{(k)} - 1]\), we have \(f_s = f := \bigl\lceil (\alpha / \beta)^{k} \tau \edmax\bigr\rceil\). Therefore, 
    \begin{equation*}
    F_{t + t^{(k)} - 1} - F_{t - 1} = t^{(k)} f \geq \frac{m(\beta - 1)}{\beta^{k+1}} = \beta^{-k}m - \beta^{-k-1}m,
    \end{equation*}
    and 
    \begin{equation*}
    \begin{aligned}
    m - F_{t + t^{(k)} - 1} &= m - F_{t - 1} - \left(F_{t + t^{(k)} - 1} - F_{t - 1}\right) \\
    & \leq \beta^{-k}m - (\beta^{-k}m - \beta^{-k-1}m) = \beta^{-k-1}m.
    \end{aligned}
    \end{equation*}
    Therefore, we must have \(m - F_{t + t^{(k)} - 1} \leq \beta^{-k-1}m\).
\end{proof}
\noindent
Note that we always have \(\beta^{-1}m \leq m - F_0 = m\).
If we consecutively apply Lemma~\ref{lemma:one_step_beta} starting with \(k = 0\), then, for \(v(k) \coloneqq \sum_{s=0}^k t^{(s)}\) we have \(m - F_{v(k) - 1} \leq \beta^{-k-1} m\). Therefore, for \(k := \frac{\log \edmax}{\log {\beta}}\), we have \(m - F_{v(k) - 1} \leq \frac{m}{\edmax}\). We can bound
\begin{equation*}
v(k) \leq \sum_{s = 0}^{\infty} t^{(k)} = \frac{\beta - 1}{\beta \tau (\alpha - 1)} \frac{m}{\edmax} \sim \frac m \edmax.
\end{equation*}
From Lemma~\ref{lemma:d_max_concentration_whp} we have \(\dmax \lesssim \edmax\) with high probability. Together with Lemma~\ref{lemma:lp_lb} this implies \(\vallp \geq \frac{m}{\dmax} \gtrsim \frac m \edmax\).
Now, if we pick \(\tilde K \coloneqq v(k) \lesssim \frac m \edmax \), we have that \(\valbgr \lesssim \frac m \edmax\). Since \(\vallp \leq \valbgr\), we have that \(\tilde K \sim \vallp\) and \(m - F_{\tilde K} \lesssim \tilde K\).

\subsection*{Case \(mp \gg \log n\)}
Here, we have that \(\edmax = mp (1 + o(1))\), therefore, picking an element that hits an average number of subsets is approximately the same as picking an element that hits close to maximum number of subsets. From the properties of the mean and the median of the binomial distribution, it follows that \(\P(\mathrm{Bin}(\tilde m, p) \geq \ceil{\tilde m p}) \geq 1/3\), for any \(\tilde m\). 

We begin with the case \(\log mp \ll \log n\). This means that \(mp\) cannot grow polynomially in \(n\), but e.g. \(mp \sim \log^2 n\) is possible. In this regime, \(\valip \sim \frac{1}{p} \log \left(\frac{mp}{\log n}\right)\). Let \(K_1 = \ceil{\frac{1}{p} \log \left(\frac{mp}{\log n}\right)}\) and \(f_1, \ldots, f_{K_1}\) be a sequence such that \(f_{s} = \ceil{mp(1-p)^s}\). Then, we have that \(m - F_{K_1} \leq m(1 - p)^{K_1} \leq \frac{1}{p}\log n\). Therefore, \((m - F_{K_1})p \sim \log n\), and we can continue with \(\tilde{f}_t\) from the previous section \(mp \sim \log n\), with \(\tilde{F}_t \coloneqq \sum_{s=1}^t \tilde{f}_t\). For this sequence \(\tilde{f}_1, \ldots, \tilde{f}_{K_2}\), we have have \(K_2 \lesssim \frac{1}{p}\), and \(m - F_{K_1} - \tilde{F}_{K_2} \lesssim \frac{1}{p} \ll \frac{1}{p} \log \left(\frac{mp}{\log n}\right)\). The required statement holds for combined sequences \(f_t\) and \(\tilde{f}_t\) and \(\tilde K \coloneqq K_1 + K_2\). \\
\noindent
Finally, we study the case \(\log mp \gtrsim \log n\), which implies that \(\valip \sim \frac{1}{p} \log n\). This case is trivial, as one can pick \(\tilde K = \ceil{\frac{1}{p} \log \left(\frac{mp}{\log n}\right)} \lesssim \valip\) and \(f_1, \ldots, f_{\tilde K}\) a sequence such that \(f_{s} = \ceil{mp(1-p)^s}\). Then, we have that \(m - F_{\tilde K} \leq m(1 - p)^{\tilde K} \leq \frac{1}{p}\log n \lesssim \valip\). 

\subsection*{From \(m - F_{\tilde K} \lesssim \tilde K\) to \(m - F_K \leq K\)}
Finally, using that \(f_t \geq 1\) by Lemma~\ref{lemma:lp_ub} unless \(F_t = m\), there exists some constant \(C > 0\), such that for \(K \coloneqq C\tilde K\), \(F_K \geq m - K\), which finishes the proof.

\end{document}